\documentclass[reqno, 11pt]{amsart}
\usepackage{amsmath,mathtools}
 \usepackage{amssymb}
\usepackage{amsthm}
\usepackage{amsmath}
\usepackage{times}
\usepackage{latexsym}
\usepackage[mathscr]{eucal}
\usepackage{graphicx}
\usepackage{float}

\numberwithin{equation}{section}
 
  \newtheorem{theorem}{Theorem}[section]
  \newtheorem{proposition}[theorem]{Proposition}
  
  \newtheorem{corollary}[theorem]{Corollary}

  \newtheorem{definition}[theorem]{Definition}
  \newtheorem{example}[theorem]{Example}
  \newtheorem{question}[theorem]{Question}

\title[The two classes of lightlike hypersurfaces]{On the structure vector field in lightlike hypersurfaces }
\author[Samuel Ssekajja]{Samuel Ssekajja}
\newcommand{\acr}{\newline\indent}
\address{ School of Mathematics\acr
 University of the Witwatersrand\acr
 Private Bag 3, Wits 2050\acr
South Africa}
\email{samuel.ssekajja@wits.ac.za} 
\thanks{}
\subjclass[2010]{Primary 53C25; Secondary 53C40, 53C50}

\keywords{Lightlike hypersurfaces, Ascreen hypersurfaces, Inascreen hypersurfaces}

\begin{document}

\begin{abstract}
We study lightlike hypersurfaces of an indefinite almost contact metric-manifold $\bar{M}$. We prove that there are only two types of such hypersurfaces, known as ascreen and inascreen, with respect to the position of the structure vector field of $\bar{M}$. We also show that the second class of hypersurfaces naturally admits an almost Hermitian structure.

\end{abstract}
\maketitle
%%%%%%%%%%%%%%%%%%%%%%%%
\section{Introduction} 
%%%%%%%%%%%%%%%%%%%%%%%%
 Lightlike hypersurfaces, $M$, of indefinite almost contact metric manifolds, $\bar{M}$,  have been studied by many authors. Among those are the following \cite{Calin2}, \cite{Duggal6}, \cite{Duggal4}, \cite{Jin4}, \cite{Jin1}, \cite{Jin2}, \cite{Jin3}, \cite{Kang} and \cite{Massamba1}. A lot of research has been done on such hypersurfaces to-date, and most of the work is on {\it tangential} hypersurfaces, that is; those which are tangent to the structure vector field of $\bar{M}$. Although it is relatively easy to study tangential lightlike hypersurfaces, it has been shown that most of the well-known lightlike hypersurfaces are non-existent in this case. For example, these hypersurfaces can not be totally umbilical (although they have been assumed to be totally umbilical in the paper \cite{Kang}), totally screen umbilical or screen conformal. This shows that the position of the structure vector field relative to each hypersurface has a great impact on the underlying geometry. In an effort to extend the study of lightlike hypersurfaces to non-tangential ones, D. H. Jin introduced a new class of lightlike hypersurfaces and named it {\it ascreen} \cite{Jin4, Jin1, Jin2, Jin3}. In this class of hypersurfaces, the structure vector field is nowhere tangent to the hypersurface. In fact, this vector field lies in the orthogonal complement of the screen distribution,  over $M$, in the ambient space. Some results have been proved regarding this class of hypersurfaces. Some of those results can be seen in the above articles. A natural question arises here;
 \begin{question}
 	Given a lightlike hypersurface $M$ of an indefinite almost contact metric-manifold $\bar{M}$.  Is it possible to precisely locate the position of the structure vector field of $\bar{M}$ relative to $M$?  
 \end{question}
 Our answer to this question is affirmative. In fact, we prove, in Theorem \ref{c8}, that there are only two types of such hypersurfaces, i.e. the ascreen and what we have named {\it inascreen}. The second class (inascreen) also includes the well-down tangential lightlike hypersurfaces. Finally, we prove that this class of hypersurfaces admits an an almost Hermitian structure (see Theorem \ref{c30}).

%%%%%%%%%%%%%%%%%%%%%%%%%%%%%%%%%%%%%%%%%%%
\section{Preliminaries} \label{pre}
%%%%%%%%%%%%%%%%%%%%%%%%%%%%%%%%%%%%%%%%%%%
An odd-dimensional semi-Riemannian manifold $(\bar{M},\bar{g})$ is called an almost contact metric-manifold \cite{Takahashi, Tanno} if there are a $(1, 1)$ tensor field $\bar{\phi}$, a vector field $\zeta$ , called structure vector field, and a 1-form $\eta$ such that
\begin{align}
\eta(\zeta)=1\quad \mbox{and} \quad \bar{\phi}^{2}=-I+\eta \otimes \zeta,\label{p1}
\end{align}
where $I$ denotes the identity transformation. From (\ref{p1}), we have 
\begin{align}\label{bg1}
\bar{\phi}\zeta=0 \quad \mbox{and}\quad \eta\circ \bar{\phi}=0. 
\end{align}
Futhermore, on an almost contact manifold $\bar{M}$ the following holds
\begin{align}\label{p2}
\bar{g}(\bar{\phi}X,\bar{\phi}Y)&=\bar{g}(X,Y)-\eta(X)\eta(Y),
\end{align}
for any $X$ and $Y$ tangent to $\bar{M}$. It follows from (\ref{p1}) and (\ref{p2}) that 
\begin{align}\label{ou}
\bar{g} (X,\zeta)=\eta(X).
\end{align}
We can, also,  see from (\ref{p1}), (\ref{bg1}) and (\ref{p2}) that $\bar{\phi}$ is skew-symmetric with respect to $\bar{g}$, i.e. 
\begin{align}\label{bg2}
	\bar{g}(\bar{\phi}X,Y)=-\bar{g}(X,\bar{\phi}Y),
\end{align}
 for any $X$ and $Y$ tangent to $\bar{M}$.

Let $(\bar{M}, \bar{g})$ be a semi-Riemannian manifold, and let $(M,g)$ be a hypersurface of $\bar{M}$, where $g=\bar{g}|_{M}$ is the induced metric tensor on $M$. We call $M$ a lightlike hypersurface if the normal bundle $TM^{\perp}$ of $M$ is a vector subbundle of the tangent bundle $TM$, of rank 1. Moreover, it is known \cite{Duggal5, Duggal6} that the complementary bundle to $TM^{\perp}$ in $TM$, called the screen distribution and denoted by $S(TM)$, is non-degenerate and the following decomposition holds; 
 \begin{align}\label{n97}
TM=S(TM)\perp TM^{\perp},
\end{align}
where $\perp$ denotes the orthogonal direct sum. A lightlike hypersurface $M$ with a chosen screen distribution will be denoted by $M=(M, g, S(TM))$.  Then, there exists a unique vector bundle $\mathrm{tr}(TM)$, called the  lightlike transversal bundle of $M$ with respect to $S(TM)$,  of rank 1 over $M$ such that for any non-zero section $\xi$ of $TM^{\perp}$ on a coordinate neighbourhood $\mathcal{U}\subset M$, there exists a unique section $N$ of $\mathrm{tr}(TM)$ on $\mathcal{U}$ satisfying the conditions 
\begin{align}
	\bar{g}(\xi,N)=1\quad\mbox{and} \quad \bar{g}(N,N)=\bar{g}(N,Z)=0,
\end{align}
for any $Z$ tangent to $S(TM)$. Consequently, we have the following decomposition.  
\begin{align}
	T\bar{M}|_{M}&=S(TM)\perp \{TM^{\perp}\oplus \mathrm{tr}(TM)\}\label{n100}\\
	&=TM\oplus \mathrm{tr}(TM),\nonumber
\end{align}
where $\oplus$ denotes a direct sum, not necessarily orthogonal. 

Let $\bar{\nabla}$ be the Levi-Civita connection of $\bar{M}$ and let $P$ be the projection morphism of $TM$ onto $S(TM)$, with respect to (\ref{n97}). Then the local Gauss-Weingarten equations of $M$ and $S(TM)$ are given by \cite{Duggal5, Duggal6}.
\begin{align}
 &\bar{\nabla}_{X}Y=\nabla_{X}Y+B(X,Y)N,\quad \bar{\nabla}_{X}N=-A_{N}X+\tau(X)N,\nonumber\\
  \mbox{and}\quad&\nabla_{X}PY= \nabla^{*}_{X}PY + C(X,PY)\xi,\quad \nabla_{X}\xi=-A^{*}_{\xi}X -\tau(X) \xi,\nonumber
 \end{align}
respectively, for all $X$ and $Y$ tangent to $M$, $\xi$ tangent to $TM^{\perp}$ and $N$ tangent to $\mathrm{tr}(T M)$.  $\nabla$ and $\nabla^{*}$ are the induced connections on $TM$ and $S(TM)$, respectively.   $B$ and $C$ are the local second fundamental forms of $M$ and $S(TM)$, respectively. Furthermore,  $A_{N}$ and $A^{*}_{\xi}$ are the shape operators of $TM$ and $S(TM)$ respectively, while $\tau$ is a 1-form on $TM$. Moreover, 
\begin{align}\nonumber
	g(A^{*}_{\xi}X,Y)=B(X,Y)\quad \mbox{and}\quad g(A_{N}X,PY)= C(X,PY), 
\end{align}
for any $X$ and $Y$ tangent to $M$. Although $\nabla^{*}$ is a metric connection, $\nabla$ is generally not, and certisfy the relation
\begin{align}\nonumber
	(\nabla_{X}g)(Y,Z)=B(X,Y)\theta(Z)+B(X,Z)\theta(Y), 
\end{align}
where $\theta$ is 1-form given by $\theta(X)=\bar{g}(X,N)$, for all $X$, $Y$ and $Z$ tangent to $M$. For more details on lightlike hypersurfaces, we refer the reader to the books \cite{Duggal5, Duggal6}.

\section{A precise location of the structure vector field $\zeta$}\label{lightlike hypersurfaces}

 Let $M$ be a lightlike hypersurface of an almost contact metric manifold  $\bar{M}$. Let $\xi$ and $N$ be the lightlike sections spanning the normal bundle $TM^{\perp}$ and transversal bundle $\mathrm{tr}(TM)$ over $M$, respectively. Since the structure vector field  $\zeta$ is tangent to $\bar{M}$, we decompose it according to (\ref{n100}) as 
\begin{align}\label{po}
	\zeta=W+a \xi+b N,
\end{align}
where $W$ is a smooth section of $S(TM)$, while $a$ and $b$ are smooth functions on $M$. Using (\ref{bg2}) and (\ref{ou}), we have 
\begin{align}
	a=\eta(N)\quad \mbox{and} \quad b=\eta(\xi).
\end{align}
Furthermore, using the first relation in (\ref{p1}), together with (\ref{ou}) and (\ref{po}), we derive $g(W,W)+2ab=1$. We say that $M$ is {\it tangent} to $\zeta$ whenever $b=\eta(\xi)=0$. In this case, C. Calin \cite{Calin2} has shown that $a=0$ too, i.e. $\zeta$ belongs to $S(TM)$.

 Let $M$ be a lightlike hypersurface of an indefinite almost contact metric manifold $\bar{M}$. By virtue of relation (\ref{bg2}), we have $\bar{g}(\bar{\phi}\xi,\xi)=-\bar{g}(\xi,\bar{\phi}\xi)$, where $\xi$ is tangent to $TM^{\perp}$. It follows that $\bar{g}(\bar{\phi}\xi,\xi)=0$ and hence the $\bar{\phi}$ is always tangent to $M$. Let us choose a screen distribution $S(TM)$ such that $\bar{\phi}\xi$ is tangent to it. Thus, by (\ref{bg2}), we have $\bar{g}(\bar{\phi}N,N)=0$ and $\bar{g}(\bar{\phi}N,\xi)=-\bar{g}(N,\bar{\phi}\xi)=0$. These relations shows that $\bar{\phi}N$ is also tangent to $M$, and in particular belonging to $S(TM)$. Furthermore, we have $g(\bar{\phi}\xi, \bar{\phi}N)=1-ab$. Hence, $\bar{\phi}TM^{\perp}$ and $\bar{\phi}\mathrm{tr}(TM)$ are vector subbundles of $S(TM)$ of rank 1. Thus, there exists a non-degenerate  distribution $D'$, such that 
\begin{align}\label{p13}
	S(TM)=\{\bar{\phi}TM^{\perp}\oplus \bar{\phi}\mathrm{tr}(TM)\}\perp D'.
\end{align}
From (\ref{n97}), (\ref{n100}) and (\ref{p13}), the decompositions of $TM$ and $T\bar{M}$ becomes
\begin{align}\label{n101}
TM&=\{\bar{\phi}TM^{\perp}\oplus \bar{\phi}\mathrm{tr}(TM)\}\perp D'\perp  TM^{\perp};\\
T\bar{M}_{|M}&=\{\bar{\phi}TM^{\perp}\oplus \bar{\phi}\mathrm{tr}(TM)\}\perp D'\perp \{TM^{\perp}\oplus \mathrm{tr}(TM)\}.\label{bg3}
\end{align}
Furthermore, from (\ref{p13}), we can decompose $W$ as 
\begin{align}\label{po1}
	W=W'+f_{1} \bar{\phi}N+f_{2}\bar{\phi}\xi,
\end{align}
where $W'$ is a smooth section of $D'$, while $f_{1}$ and $f_{2}$ are smooth functions on $M$. 

\begin{proposition}
	$\bar{\phi}D'\subset S(TM)$.
\end{proposition}
\begin{proof}
	On one hand, using (\ref{bg2}) and (\ref{p13}), we derive 
	\begin{align}\label{c3}
		\bar{g}(\bar{\phi}X', \xi)=-\bar{g}(X', \bar{\phi}\xi)=0,
	\end{align}
	for every $X'$ tangent to $D'$. Relation (\ref{c3}) shows that $\bar{\phi}X'$ is tangent to $M$. On the other hand, we have 
	\begin{align}\nonumber
		\bar{g}(\bar{\phi}X', N)=-\bar{g}(X', \bar{\phi}N)=0,
	\end{align}
which, indeed, shows that $\bar{\phi}X'$ is tangent to $S(TM)$.
\end{proof}
\begin{proposition}\label{c9}
	$D'$ is $\bar{\phi}$-invariant, i.e. $\bar{\phi}D'\subseteq D'$, if and only if one of the following holds: 
	\begin{enumerate}
		\item $a=b=0$;
		\item $W'=0$.
	\end{enumerate}
\end{proposition}
\begin{proof}
	In view of the first relation in (\ref{p2}), we have 
	\begin{align}\label{c11}
		\bar{g}(\bar{\phi}X', \bar{\phi}N)=-a\eta(X')\quad \mbox{and}\quad \bar{g}(\bar{\phi}X', \bar{\phi}\xi)=-b\eta(X'),
	\end{align}
	for every $X'$ tangent to $D'$. Now, if $D'$ is $\bar{\phi}$-invariant then (\ref{c11}) gives 
	\begin{align}\label{c12}
		a\eta(X')=b\eta(X')=0.
	\end{align}
Relation (\ref{c12}) shows that either $a=b=0$ or $\eta(X')=g(X', W')=0$, that is $W'=0$. The converse is obvious.
\end{proof}
\begin{definition}
\rm{
	A lightlike hypersurface $M$ of an indefinite almost contact metric-manifold $\bar{M}$ is said to be an {\it ascreen} \cite{Jin4,Jin1,Jin2, Jin3} lightlike hypersurface  of $\bar{M}$ if the vector field $\zeta$ belongs to $S(TM)^{\perp}=TM^{\perp} \oplus \mathrm{tr}(TM)$.
	}
\end{definition}
\noindent In any ascreen lightlike hypersurface, $W'=0$ and $a,b\ne 0$. It follows that $D'$ is $\bar{\phi}$-invariant. Moreover, the following result about ascreen hypersurfaces is known.
\begin{theorem}[\cite{Jin4}]\label{c31}
Let $M$ be a lightlike hypersurface of an indefinite almost contact metric-manifold $\bar{M}$. Then $M$ is an ascreen lightlike hypersurface of $\bar{M}$ if and only if $\bar{\phi}TM^{\perp}=\bar{\phi}\mathrm{tr}(TM)$.	
\end{theorem}
	
\noindent In the next theorem, we show that there are only two types of lightlike hypersurface $M$ of an indefinite almost contact manifold $\bar{M}$ according to the position of the structure vector field $\zeta$.

\begin{theorem}\label{c8}
	Let $M$ be a lightlike hypersurface of an indefinite almost contact metric-manifold $\bar{M}$. Then, $\zeta$ takes exactly one of the following forms:
	\begin{enumerate}
		\item $\zeta=a\xi+bN$, i.e. $M$ is an ascreen hypersurface;
		\item $\zeta=W'+a\xi+bN$, where $W'$ is a non-zero section of $D'$.
	\end{enumerate}
	
\end{theorem}
\begin{proof}
\noindent From relations (\ref{po}) and (\ref{po1}), $\zeta$ can be written as 
\begin{align}\label{jk}
	\zeta=W'+f_{1}\bar{\phi}N+f_{2}\bar{\phi}\xi+a\xi+bN.
\end{align}
Taking the inner product of (\ref{jk}) with $\bar{\phi}\xi$ and $\bar{\phi}N$ in turns, and considering (\ref{p2}) and (\ref{bg1}), we have 
\begin{align}
	b^{2}f_{2}=f_{1}(1-ab)\quad \mbox{and}\quad a^{2}f_{1}=f_{2}(1-ab),\label{mu2}
\end{align}
respectively. It follows from (\ref{mu2}), that 
\begin{align}\label{c4}
	(2ab-1)f_{1}=0 \quad \mbox{and}\quad (2ab-1)f_{2}=0.
\end{align}
From (\ref{c4}), we have the following cases:
\begin{enumerate}
	\item $2ab=1$: 
	Taking the inner product on both sides of (\ref{jk}) with $\zeta$, we get $g(W', W')+2ab=1$. Thus, we have $g(W', W')=0$, which implies that $W'=0$ by the fact that $D'$ is non-degenerate. Then $\zeta$ in (\ref{jk}) reduces to 
\begin{align}\label{gf}
	\zeta=f_{1}\bar{\phi}N+f_{2}\bar{\phi}\xi+a\xi+bN.
\end{align}
Applying $\bar{\phi}$ to (\ref{gf}), and remembering that $\bar{\phi}\zeta=0$, we get 
\begin{align}\label{bv1}
	a\bar{\phi}\xi+b\bar{\phi}N-f_{2}\xi-f_{1}N+f\zeta=0,
\end{align}
where $f=af_{1}+bf_{2}$. Then using (\ref{gf}) and (\ref{bv1}), we get 
\begin{align}
	(ff_{2}+a)\bar{\phi}\xi+&(ff_{1}+b)\bar{\phi}N=0,\label{bv2}\\
	af&=f_{2},\label{bv3}\\
	bf&=f_{1}.\label{bv4}
\end{align}
Note that none of the functions $ff_{2}+a$ and $ff_{1}+b$ seen in (\ref{bv2}) vanishes. In fact, if one assumes that $ff_{2}+a=0$, then (\ref{bv3}) leads to 
\begin{align}\label{c5}
	aff_{2}+a^{2}=f^{2}_{2}+a^{2}=0.
\end{align}
Relation (\ref{c5}) shows that $f_{2}=a=0$. But this is impossible since $2ab=1$. Also, if $ff_{1}+b=0$, relation (\ref{bv4}) leads to $f_{1}=b=0$, which is again impossible. Thus, we may rewrite (\ref{bv2}) as 
\begin{align}\label{c6}
	\bar{\phi}\xi=\lambda \bar{\phi}N,
\end{align}
where $\lambda=-\frac{ff_{1}+b}{ff_{2}+a}$ is a non-zero smooth function on $M$. It is clear from (\ref{c6}) that $\bar{\phi}TM^{\perp}=\bar{\phi}\mathrm{tr}(TM)$, and hence by  Theorem \ref{c31}, $M$ is ascreen.
\item $2ab\ne 1$: In this case, we see that $f_{1}=f_{2}=0$ and therefore $\zeta=W'+a\xi+bN$,
from which we get $g(W',W')+2ab=1$. Furthermore, $W'\ne 0$ since if one takes $W'=0$, we get $2ab=1$, which is a contradiction,
\end{enumerate}
which completes the proof.
\end{proof}

\begin{corollary}
	Every lightlike hypersurface of an indefinite almost contact metric-manifold  in which $D'$ is $\bar{\phi}$-invariant, i.e. $\bar{\phi}D'\subseteq D'$, is either ascreen or tangential with $\zeta$ tangent to $D'$.
\end{corollary}
\noindent Using part (2) of Theorem \ref{c8}, we have the following definition.
\begin{definition}\label{c13}
\rm{
A lightlike hypersurface $M$ of an definite almost contact metric-manifold $\bar{M}$ is called {\it inascreen} if $W'\ne 0$. In addition, if $b\ne 0$ then $M$ will be called a {\it proper inascreen} lightlike hypersurface.
}
\end{definition}
\begin{example}
\rm{
Any lightlike hypersurface of an indefinite almost contact metric-manifold that is tangent to the structure vector field $\zeta$ is inascreen with $a=b=0$.
}
\end{example}
\noindent In view of Proposition \ref{c9} and Definition \ref{c13}, we hve the following. 
\begin{proposition}
	The only inascreen lightlike hypersurfaces of an indefinite almost contact metric-manifold with a $\bar{\phi}$-invariant $D'$ are those tangent to the structure vector field $\zeta$. 
\end{proposition}

\section{Proper inascreen lightlike hypersurfaces}\label{parasff}
 Since inascreen lightlike hypersurfaces with $b=0$ are the well-known tangential lightlike hypersurfaces, in this section we shall focus only on the proper ones, i.e. those in which $b\ne 0$.

\begin{proposition}\label{c17}
Let $M$ be a proper inascreen lightlike hypersurface of an indefinite almost contact metric-manifold $\bar{M}$ then, the  following holds:
\begin{enumerate}
\item $D'$ is never a $\bar{\phi}$-invariant distribution;
\item $\bar{\phi}N$ and $\bar{\phi}\xi$ are linearly independent vector fields;
\item  $\zeta$ is linearly independent with any $X$ tangent to $M$.
\end{enumerate}
\end{proposition}
\begin{proof}
In a proper inascreen lightlike hypersurface, we know that $W'\ne 0$ and $b\ne 0$. It follows from Proposition \ref{c9} that $D'$ is never $\bar{\phi}$-invariant. Suppose that 
\begin{align}\label{c40}
	l_{1}\bar{\phi}N+l_{2}\bar{\phi}\xi=0,
\end{align}
for some smooth functions $l_{1}$ and $l_{2}$ on $M$. Taking the inner product of (\ref{c40}) with $\bar{\phi}N$ and $\bar{\phi}\xi$ by turns, we get 
\begin{align}\label{c41}
	a^{2}l_{1}=l_{2}(1-ab)\quad \mbox{and}\quad b^{2}l_{2}=l_{1}(1-ab),\end{align}
respectively. The two relations in (\ref{c41}) leads to 
\begin{align}\label{c42}
	l_{1}(1-2ab)=0\quad\mbox{and}\quad l_{2}(1-2ab)=0.
\end{align}
Since $2ab\ne 1$, (\ref{c42}) gives $l_{1}=l_{2}=0$. Finally, suppose that 
\begin{align}
	l_{3} X+l_{4}\zeta=0,
\end{align}
for any tangent vector field $X$, where $l_{3}$ and $l_{4}$ are some smooth functions on $M$. Then, taking the inner product of this relation with $\xi$ leads to $l_{4}\eta(\xi)=l_{4}b=0$. Since $b\ne 0$, we get $l_{4}=0$. Consequently, $l_{3} X=0$, which gives $l_{3}=0$. 
\end{proof}
\noindent Using part (2) of Proposition \ref{c17}, we deduce the following.
\begin{proposition}
	Let $M$ be a proper inascreen lightlike hypersurface of an indefinite almost contact metric-manifold $\bar{M}$ then, $\dim M\ge  4$ and $\dim \bar{M}\ge 5$.
\end{proposition}
\noindent As $\zeta$ is nowhere tangent to $M$, we put  
\begin{align}\label{bg34}
	\bar{\phi}X=\phi X+\omega(X)\zeta,
\end{align}
for any $X$ tangent to $M$. Here, $\phi X$ is the tangential part (with respect to $\zeta$) of $\bar{\phi}X$ to $M$. It is easy to see that $\phi$ and $\omega$ are tensor fields of type $(1,1)$ and $(0,1)$, respectively, on $M$. We say that $M$ is {\it invariant} if $\omega=0$. Since $\bar{\phi}\xi$ is tangent to $M$, it follows from (\ref{bg34}) that 
\begin{align}\label{bg35}
	\bar{\phi}\xi=\phi \xi\quad \mbox{and} \quad \omega(\xi)=0.
\end{align}
Applying $\bar{\phi}$ to (\ref{bg34}), and using (\ref{p1}) and (\ref{bg1}), we have 
\begin{align}\label{bg36}
	-X+\eta(X)\zeta=\phi^{2}X+\omega(\phi X)\zeta.
\end{align}
Comparing components in (\ref{bg36}) we get 
\begin{align}
	\phi ^{2}X&=-X \quad\mbox{and}\quad \omega(\phi X)=\eta(X),\label{bg38}
\end{align}
for any $X$ tangent to $M$. Thus, the tensor $\phi$ is an {\it almost complex structure} on $M$.

\begin{proposition}\label{bg50}
	There exist no invariant inascreen lightlike hypersurface of an indefinite almost contact metric-manifold $\bar{M}$.
\end{proposition}
\begin{proof}
	$M$ is invariant whenever $\omega=0$. It follows from (\ref{bg38}) that $\eta(X)=0$, for any $X$ tangent to $M$. Taking $X=\xi$ in the last relation, we get $0=\eta(\xi)=b$, which is a contradiction to $b\ne 0$.
\end{proof}

\noindent By a direct calculation, while using (\ref{p2}), (\ref{bg2}) and (\ref{bg34}), we derive 
\begin{align}
	g(\phi X,\phi Y)&=g(X,Y)-\eta(X)\eta(Y)+\omega(X)\omega(Y),\label{bg40}\\ 
	g(\phi X,Y)+\omega(X)\eta(Y)&=-g(X,\phi Y)-\omega(Y)\eta(X),\label{bg41}
\end{align}
for any $X$ and $Y$ tangent to $M$.

\begin{proposition}
	There exit no proper inascreen lightlike hypersurface of an indefinite almost contact metric-manifold such that; 
	\begin{enumerate}
		\item $(g, \phi)$ is an almost Hermitian structure on $M$;
		\item $\phi$ is skew-symmetric with respect to $g$,
	\end{enumerate}
	for any $X$ and $Y$ tangent to $M$.
\end{proposition}
\begin{proof}
	Suppose that $(g, \phi)$ is an almost Hermitian structure on $M$. Then $g$ must satisfy the relation $g(\phi X,\phi Y)=g(X,Y)$, for any $X$ and $Y$ tangent to $M$. In view of  (\ref{bg40}), this implies that 
	\begin{align}\label{c19}
		\eta(X)\eta(Y)=\omega(X)\omega(Y),
	\end{align}
	for any $X$ and $Y$ tangent to $M$. With $X=Y=\xi$ in (\ref{c19}), and considering relation  (\ref{bg35}), we get 
	\begin{align}\nonumber
		b^{2}=\eta(\xi)\eta(\xi)=\omega(\xi)\omega(\xi)=0,	
	\end{align}
	which is impossible for a proper inascreen lightlike hypersurface. Next, suppose that $\phi$ is skew-symmetric with respect to $g$. Then relation (\ref{bg41}) implies that 
	\begin{align}\label{c18}
		\omega(X)\eta(Y)=-\omega(Y)\eta(X),
	\end{align}
	for any $X$ and $Y$ tangent to $M$. Taking $Y=\xi$ in (\ref{c18}) and considering (\ref{bg35}), we have $\omega(X)b=0$. Since $b\ne 0$, we get $\omega=0$ meaning that $M$ is invariant. However, this is not possible by Proposition \ref{bg50}.
\end{proof}
\noindent Let us set $\tilde{g}=g+\omega\otimes \omega$. This metric $\tilde{g}$ is degenerate on $M$ since $\tilde{g}(X,\xi)=0$ and $\tilde{g}(X,\phi\xi)=0$, for every $X$ tangent to $M$. Moreover, using (\ref{bg38}) and (\ref{bg40}) we have 
\begin{align}\label{bg52}
	\tilde{g}(\phi X,\phi Y)&=g(\phi X,\phi Y)+\omega(\phi X)\omega(\phi Y)\nonumber\\
	&=g(X,Y)-\eta(X)\eta(Y)+\omega(X)\omega(Y)+\omega(\phi X)\omega(\phi Y)\nonumber\\
	&=g(X,Y)+\omega(X)\omega(Y)\nonumber\\
	&=\tilde{g}(X,Y).
	\end{align}
Therefore, in view of (\ref{bg52}), we have the following.
\begin{theorem}\label{c30}
	Every proper inascreen lightlike hypersurface of an indefinite almos contact metric-manifold admits an almost Hermitian structure $(\tilde{g}, \phi)$.
\end{theorem}

\end{document}